\documentclass[a4paper,12pt]{amsart}
\usepackage[utf8]{inputenc}
\usepackage{cleveref}

\usepackage{amsmath}
\usepackage{MnSymbol}
\usepackage{wasysym}
\usepackage{amsmath}

\usepackage{tikz}
\usepackage{mathdots}
\usepackage{yhmath}
\usepackage{cancel}
\usepackage{color}
\usepackage{siunitx}
\usepackage{array}
\usepackage{multirow}
\usepackage{gensymb}
\usepackage{tabularx}
\usepackage{booktabs}
\usetikzlibrary{fadings}
\usetikzlibrary{patterns}
\usetikzlibrary{shadows.blur}

\author{Alfredo R. Freire}
\thanks{The author wants to thanks Joel Hamkins for fruitful discussions on relations between finite set theory and arithmetic during the time under his supervision, Hugo Mariano, Rodrigo Freire and Vincent Peluce for many suggestion on the organization of the paper and for their detailed revision. The author has been supported by BEPE-FAPESP grant 2017/21020-0 during the time the ideas in this paper were developed.}
\title{Does set theory really ground arithmetic truth?} 

\newtheorem{defi}{Definition}
\newtheorem{theo}{Theorem}

\newtheorem{coro}{Corollary}

\newtheorem{prop}{Proposition}

\newcommand{\so}{\rightarrow}
\renewcommand{\see}{\leftrightarrow}
\newcommand{\M}{\mathcal{M}}
\newcommand{\N}{\mathcal{N}}

\newcommand{\Lg}{\mathcal{L}}

\newcommand{\pair}[1]{\langle #1 \rangle}

\begin{document}
\maketitle

\begin{abstract}
We consider the foundational relation between arithmetic and set theory. 
Our goal is to criticize the construction of standard arithmetic models as providing grounds for arithmetic truth (even in a relative sense). 
Our method is to emphasize the incomplete picture of both theories and treat models as their syntactical counterparts. 
Insisting on the incomplete picture will allow us to argue in favor of the revisability of the standard model interpretation. 
We then show that it is hopeless to expect that the relative grounding provided by a standard interpretation can resist being revisable.
We start briefly characterizing the expansion of arithmetic `truth' provided by the interpretation in a set theory.
Further, we show that, for every well-founded interpretation of recursive extensions of PA in extensions of ZF, the interpreted version of arithmetic has more theorems than the original. This theorem expansion is not complete however.
We continue by defining the coordination problem. 
The problem can be summarized as follows. 
We consider two independent communities of mathematicians responsible for deciding over new axioms for ZF and PA. 
How likely are they to be coordinated regarding PA’s interpretation in ZF?
We prove that it is possible to have extensions of PA not interpretable in a given set theory ST. 
We further show that the probability of a random extension of arithmetic being interpretable in ST is zero.
\end{abstract}

\section{Introduction}

In this article we study the idea of reducing Arithmetic to Set theory as a strategy for grounding arithmetic truth. The method of reduction we have in mind is \textbf{interpretation}. A brief discussion about this method will be provided in the next section. All proofs provided in this article are elementary.
Our purpose is to raise questions and suggest an different view of this situation, rather then hard mathematical results.

\begin{defi}
An \textbf{interpretation} of $\Lg_{T_1}$ in $\Lg_{T_2}$ is a mapping $I$ of formulas of $\Lg_{T_1}$ in $\Lg_{T_2}$ such that: 

If $\alpha, \beta \in \Lg_{T_1}$ and $P$ is a n'ary predicate in $\Lg_{T_1}$, then the transformation $*$ is such that
\begin{enumerate}
	\item $(\lnot \alpha)^* = \lnot \alpha^*$.
	\item $(\alpha \lor \beta)^* = \alpha^* \lor \beta^*$.
	\item $P^*$ is a fixed formula of $T_2$ with at most n free variables.
	\item $(\forall x (\alpha))^* = \forall x (U(x) \rightarrow \alpha^*)$, being $U$ a fixed formula with one free variable in $T_2$.
\end{enumerate}

Being $x_1, x_2, \ldots, x_n$ all free variables occurring in $\alpha$,

\begin{equation*}
(\alpha)^I = U(x_1) \land U(x_2) \land \ldots \land U(x_n) \so \alpha^*
\end{equation*}

We say $I$ is an interpretation of the theory $T_1$ in the theory $T_2$ if for every theorem $\varphi$ of $T_1$, the formula $\varphi^I$ is a theorem of $T_2$.
\end{defi}

We start our investigation by studying the capability of models to provide grounds for a theory. Even though our conceptualization of models is such that each formula is either satisfied or not by a given model, our ability to determine which option is the case is limited. The reason is that if the only thing we know about $V$ is that it satisfies ZF, then we can determine that $V \vDash \varphi$ if, and only if, $ZF \vdash \varphi$.

This is the reason why we may consider models as their syntactical representation via interpretations. Each model definable in a given base model $V \vDash ZF$ can be said be to the result of bounding the elements of $V$ to a given interpretation $I$. By doing so, we can keep in mind our limited knowledge of the models. Since, if $\M$ is definable in $V$ (i.e $\M = I^V$) and we do not know any other information about $V$ other than it satisfies ZF, then 
\begin{equation}
\text{We know } \M \vDash \varphi 
\text{ if, and only if, }
ZF \vdash \varphi^I
\end{equation}

Further, we investigate the grounding relation represented by interpreting PA in ZF.  Notably, the standard interpretation expands what may be considered true for arithmetic -- i.e. many independent formulas in PA become theorems as we see them in ZF through the interpretation. We show that this expansion occurs for any well-founded interpretation between PA and ZF. 

\begin{theo}\label{thm:expansion}
Every well founded interpretation $I$ of a recursive extension A of PA in an extension S of ZF is such that there is an undecidable formula $\varphi$ in A that is interpreted as theorem in S under the interpretation.
\end{theo}

\newtheorem*{expansion}{Theorem~\ref{thm:expansion}}

But, even though we expect that interpretations of PA in ZF expand arithmetical truth, an extension of ZF does not completely decide on arithmetical formulas:

\begin{theo}\label{thm:incomplete}
For any recursive extension S of ZF and any interpretation I (for instance, the standard interpretation) there is an arithmetical formula that S does not decide under the interpretation.
\end{theo}

\newtheorem*{incomplete}{Theorem~\ref{thm:incomplete}}

At any stage in the development of ZF (a recursive extension), the concept of arithmetical truth will still be open. This is a consequence of the theorem \ref{thm:incomplete}, as some arithmetic formulas will be undecidable under the interpretation in the extended set theory. Hence it is possible to build two structures satisfying the set theory that disagree about the truth value of an arithmetic formula.

It is due to this phenomena that we consider what I call \textbf{the coordination problem}: consider that there are two groups of mathematicians responsible for deciding over new axioms. The first will decide over axioms for arithmetic and the second for set theory. How should we consider the relation between the two groups? Note that if we consider that the arithmetic group should conform to any development provided by the set theory group, it becomes hard to see in what sense the interpretation of arithmetic into set theory have any foundational role. This framework is indistinguishable from simply taking arithmetic to live in set theory.

Therefore, it is important to consider the possibility of the coordination between the two theories to break. Is it possible that an extension of arithmetic not to be interpretable in any extension of a given extension of set theory? We show this is the case with the result:

\begin{theo}\label{thm:noninterpretable}
For any recursive extension A of PA and given that S is an extension of ZF, there is a extension A$^+$ of A that is not interpretable in S.
\end{theo}

\newtheorem*{noninterpretable}{Theorem~\ref{thm:noninterpretable}}

\begin{theo}\label{thm:probabilityzero}
For every consistent extension S of ZF, the probability that a random consistent extension of a S-standard version of PA is interpretable in S is zero.
\end{theo}

\newtheorem*{probabilityzero}{Theorem~\ref{thm:probabilityzero}}

It is indeed possible to generate the extension A$^+$ for any given S. But, how likely is it to be the case? If we consider that any consistent extension of arithmetic is equally likely at any stage in the development of the arithmetical group, then the probability that a development of arithmetic is interpretable in a given extension of ZF is zero.

\section{The standard model of arithmetic}

The strategy of offering set theoretical models for describing objects of a theory comes from the works of Tarski, Mostowsky, and Robinson in the 1940s \cite{tarski1953}. Ever since this date, mathematicians and philosophers often resort to this strategy. It is generally accepted that once we start talking about models, we put aside the formal aspects of the mathematical subject and start talking about its objects and truths. Nevertheless, because of G\"odel's incompleteness theorem and L\"owenhein-Skolem theorem, there is no formal way to fix the model of any recursive extension of Peano arithmetic. It is impossible to say that the only model that satisfies our descriptions of arithmetic is the intended model, no matter how extensively we describe arithmetic. Still, using a set theoretical apparatus, we can describe the intended model as $N = \pair{\omega, +,., 0, s}$ (called the \textit{standard model}). We can then show that a set theory like ZF is expressive enough to define a truth predicate for this interpretation.

The literature on this subject generally presents two approaches for fixing the standard model: (i) one should offer extra-logical (or second-order) reasons for choosing $N$ from the myriad possible models for arithmetic; (ii) one should abandon the model-theoretical construction and find other ways to ground arithmetic truth. A renewed version of (ii) can be seen in Gabbay's defense of a new kind of formalism \cite{gabbay2010formalist}; Moreover, others may abandon a privileged emphasis on $N$, because we must focus on mathematical practice (Ferreirós \cite{ferreiros2015mathematical}) or because we must commit ourselves to a realistic multiverse (Hamkins \cite{hamkins2012set}). Still, differences of opinion are more common as to how and why we should follow project (i). Those like Williamson \cite{williamson2016absolute} argue for metaphysical reasons for setting $N$, others like Maddy \cite{maddy2014second}, Quine \cite{quine1964ontological} or Putnam \cite{putnam1967mathematics} advocate ways to naturalize the reasons for $N$. Finally, a recent approach grounds $N$ in the mathematical practice using a normative basis in place of the Platonist commitment with $N$ \cite{freire2012existence}.

The question of the adequacy of $N$ is often overlooked. The assumption behind this is that if something is a model of arithmetic, then it is $N$. We may not know why this is the intended model or even deny that such a model exists, but conformity to $N$ is hardly questioned. However, presenting $N$ as an object without further consideration is a category mistake. Notably, a similar category mistake would be to say that `it has been two sun revolutions since so and so'. The phrase `two sun revolutions' is used as quantity of time, even though it describes a movement in reference to the sun. Hence, the statement would be a category mistake unless an implicity reference to Earth is assumed -- and not Mars, for instance.  Precisely stated, $N$ is an interpretation of PA in the language of membership. It represents therefore a construction of objects for arithmetic in terms of objects of a given set theory. Hence, it is only when we fix the objects for a set theory that the objects expressed in the construction $N$ gain life. This idea is what we call from now on \textbf{relative grounding}.

For any given model of set theory $V \vDash ZF$, the interpretation $N$ can be understood as a procedure for obtaining a model $\N$ for PA. The model $\N = \pair{Obj, +, ., 0, s}$ is build from the interpretation $N = \pair{U, f_+, f_., f_s, Zero}$ as follows:
\begin{enumerate}
	\item $Obj = \{x \in V \mid V \vDash U(x)\}$.
	\item $0^\N = a$ such that $V \vDash Zero(a)$.
	\item $+^\N = \{\pair{x, y, z} \mid x, y, z \in Obj \text{ and } 
					V \vDash f_+(x, y) = z\}$.
	\item $.^\N = \{\pair{x, y, z} \mid x, y, z \in Obj \text{ and } 
					V \vDash f_.(x, y) = z\}$.
	\item $s^\N = \{\pair{x, y} \mid x, y \in Obj \text{ and } 
					V \vDash f_s(x) = y\}$.
\end{enumerate}

Our idea is to insist on the incomplete picture of the set theoretical representation of arithmetic. We note that $V$ suffers from the same problem as $N$, for it is based in a incomplete theory ZF. Therefore, the picture of arithmetic obtained from reducing PA to $V$ by $N$ may also be incomplete. 

So to what are we committing in case we say $N$ is the standard model of arithmetic? It seems like the single construction for the intended model of arithmetic is based on the idea condensed in the sentence: ``no matter which model of set theory one is assuming, the model of arithmetic would be given by $N$''. Indeed, the picture provided by the literature is that of \textit{\textbf{revisable} truth for set theory and arithmetic -- but \textbf{unrevisable} reduction of arithmetic in set theory}. In the next sections, we argue that for taking the standard model to have a foundational role one should assume the interpretation to be revisable. For now, we consider the characterization of arithmetic in set theory in more details.

\subsection{Foundational characterization of PA in ZF} 

Being N the standard interpretation of arithmetic in ZF, we call the set $A^{ZF}_N = \{\varphi \mid ZF \vdash \varphi^N\}$ the expansion of arithmetic truth under the interpretation. Indeed some undecidable formulas $\varphi$ of PA are `true' in the standard model ($ZF \vdash \varphi^N$). This is the case for the G\"odel formula, Goodstein's theorem and many others. We will thus consider more broadly the question of expansion of arithmetic truth from interpretations in set theories. 

Given that $I$ is an interpretation of an arithmetic A in a set theory S and $Th(A) = \{\varphi \mid A \vdash \varphi\}$, we expect to have $Th(A)  \subsetneqq A^{S}_I \subsetneqq$ Arithmetic truth, as we see in the figure:

\vspace{0,5cm}
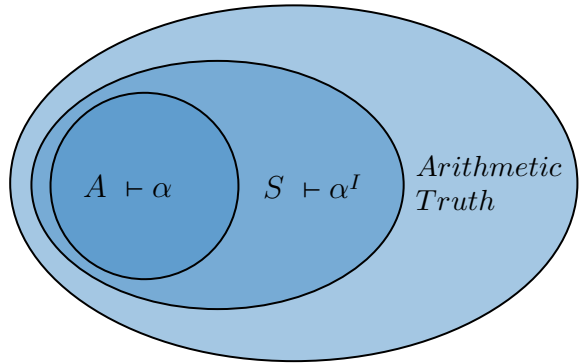
\begin{center}
\tikzset{every picture/.style={line width=0.75pt}} 

\begin{tikzpicture}[x=0.75pt,y=0.75pt,yscale=-1,xscale=1]
\draw  [fill={rgb, 255:red, 74; green, 144; blue, 200 }  ,fill opacity=0.5 ] (191.77,119.98) .. controls (191.77,70.55) and (255.12,30.48) .. (333.28,30.48) .. controls (411.44,30.48) and (474.79,70.55) .. (474.79,119.98) .. controls (474.79,169.4) and (411.44,209.47) .. (333.28,209.47) .. controls (255.12,209.47) and (191.77,169.4) .. (191.77,119.98) -- cycle ;
\draw  [fill={rgb, 255:red, 74; green, 144; blue, 200 }  ,fill opacity=0.5 ] (202.37,120.85) .. controls (202.37,86.37) and (243.95,58.41) .. (295.24,58.41) .. controls (346.52,58.41) and (388.1,86.37) .. (388.1,120.85) .. controls (388.1,155.33) and (346.52,183.28) .. (295.24,183.28) .. controls (243.95,183.28) and (202.37,155.33) .. (202.37,120.85) -- cycle ;
\draw  [fill={rgb, 255:red, 74; green, 144; blue, 200 }  ,fill opacity=0.5 ] (211.9,122.03) .. controls (211.52,96.14) and (232.2,74.85) .. (258.09,74.47) .. controls (283.98,74.09) and (305.28,94.77) .. (305.66,120.66) .. controls (306.04,146.55) and (285.36,167.84) .. (259.47,168.22) .. controls (233.58,168.6) and (212.28,147.92) .. (211.9,122.03) -- cycle ;
\draw (250.39,121.41) node  [font=\normalsize]  {$A\ \vdash \alpha$};
\draw (342.7,121.41) node  [font=\normalsize]  {$S\ \vdash \alpha^{I}$};
\draw (430.48,121.41) node  [font=\small]  {$ \begin{array}{l}
Arithmetic\ \\
Truth
\end{array}$};
\end{tikzpicture}
\end{center}
\vspace{0,5cm}

We start considering the expansion of arithmetic truth in case we have a well founded interpretation.

\begin{defi}
Let $x < y$ be the arithmetical relation $\exists z (z \neq 0 \land x + z = y)$.
The interpretation $I$ of an arithmetic $A$ in a set theory $S$ is well founded if $S$ proves that (i) for every subset $x$ of $U_I$, there is a $<$-minimal element in $x$ and that (ii) for every $y$, $\{z \mid z < y\}$ is a set.
\end{defi}

\begin{expansion}
Let $A$ be a consistent recursive extension of PA and $S$ a consistent extension of ZF. We further assume that there is a well founded interpretation I of A in S. Then there is a formula $\varphi$ which is undecidable in $A$ such that $S \vdash \varphi^I$. In other words, arithmetical truth is expanded under the interpretation I of A in S.
\end{expansion}

\begin{proof}
Kaye and Wong prove that PA is bi-interpretable with finite set theory ($ZF_{fin}$) in \cite{kaye2007interpretations}\footnote{To understand this proof, it is sufficient to know that $T_1$ and $T_2$ are bi-interpretable if (i) both theories interpret each other and (ii) the composition of the interpretation is equivalent to an identity interpretation. An extensive treatment of the bi-interpretation phenomenon in set theories can be found in \cite{hamkinsFreire2020biinterpretation}.}.
Since $PA$ is bi-interpretable with $ZF_{fin}$, there is a recursive extension $S'$ of $ZF_{fin}$ bi-interpretable by $B$ with $A$. Lets suppose $S$ is conservative for $A$ under the interpretation $I$:
\begin{equation}\label{conserv}
S \vdash \varphi^I \text{ if, and only if, } A \vdash \varphi
\end{equation}
Then we can obtain an interpretation $J$ of $S'$ in $S$ such that
\begin{equation}
S \vdash \varphi^J \text{ if, and only if, } S' \vdash \varphi
\end{equation}
Since the bi-interpretation $B$ is well founded and $I$ is well founded, $J$ is also well founded. Thus, by Mostowski collapse, we have $J$ is isomorphic with the interpretation $\pair{M, \in}$ with $M$ a transitive class. We note that $M \subseteq V_\omega$ in $S$, for otherwise we would have an infinite member $a$. In turn, this implies in $S'$ the contradiction $\bigcup \{rank(x) \mid x \in a\}$ is inductive.

Notably, as $S'$ is consistent, $V_\omega$ satisfies the predicate $Con(S')$. It follows that $S \vdash Con^{V_\omega}(S')$. But this is absurd, for it would imply the contradiction $S' \vdash Con(S')$. Therefore, the statement (\ref{conserv}) is false. As $A \vdash \varphi$ implies $S \vdash \varphi^I$ by the interpretation, there is a formula $\gamma$ such that $S \vdash \gamma^I$ and $\gamma$ is undecidable in $A$.
\end{proof}

Another venue to consider the problem is to guarantee that the system S realizes whether it is or not an expansion of A under the interpretation.

\begin{prop}
Let $A$ be a consistent recursive extension of PA and $S$ a consistent extension of ZF. We further assume that there is a recursive process $\delta$ definable in S that enumerates formulas that satisfies $S \vdash \varphi^I$ implies $A \vdash \varphi$. Then there is a formula $\varphi$ which is undecidable in $A$ such that $S \vdash \varphi^I$.
\end{prop}

\begin{proof}
Take $\varphi$ such that $PA \nvdash \lnot \varphi$, then $PA \cup \{\varphi\}$ is consistent from completeness theorem in S. Thus, for $S \vdash Con(PA)$, any finite extension $PA + \varphi$ is such that $S \vdash Con(PA + \varphi)$. Further, for A is a recursive extension of PA, we have that any finite subset $\Delta$ of A is such that $S \vdash Con(\Delta)$. So $S \vdash Con(A)$ from compactness theorem.

Lets suppose $I$ is an interpretation of A in S and further that 
\begin{equation}\label{conservative}
S \vdash \varphi^I \text{ if, and only if, } A \vdash \varphi.
\end{equation}
Thus, from the enumeration $\delta$, we may internalize the argument in ($\ref{conservative}$) as $S \vdash \ulcorner(\ref{conservative})\urcorner$ -- since the enumeration of the converse is given by the interpretation.

We note that (\ref{conservative}) implies $Con(A) \see Con(S)$. So, by also internalizing this argument in S, we obtain 
\begin{equation}
S \vdash \ulcorner (\ref{conservative}) \urcorner \so (Con(A) \see Con(S)).
\end{equation}
From G\"odel's incompleteness theorem, we have that $S \nvdash Con(S)$. Therefore
\begin{equation}\label{SdoesnotproveconsistencyofA}
S \nvdash \ulcorner (\ref{conservative}) \urcorner
\end{equation}
And this is a contradiction.

Thus the equation \ref{conservative} is false. However, we know that $A \vdash \varphi$ implies $S \vdash \varphi^I$ from the interpretation. We conclude that there is a formula undecidable $\varphi$ in $A$ such that $S \vdash \varphi^I$.
\end{proof}

A complete answer to the problem is still open. Is it possible to build an interpretation of recursive and consistent extensions S and A such that there is an interpretation of A in S preserving A's truth? We believe not. And the results presented indicate that this may not be possible.

We have seen that interpretations of arithmetic in set theories generally expand what may be taken to be arithmetical truth ($Th(A)  \subsetneqq A^{S}_I$). Yet this expansion is not necessarily complete ($A^{S}_I = $ arithmetic truth). A confusion in this regard is due to the idea that model constructions in set theories offer venues for defining truth for interpreted theories. Each interpretation I represents the appropriate model construction such that the grounding set theory can provide the notion of satisfaction $I \vDash \varphi$ for any formula. Eventually, we would have that for any formula $\gamma$, either $I \vDash \gamma$ or $I \vDash \lnot \gamma$. However, a more syntactical approach make it clear that this is simply the expression of the \textit{excluded middle}. Indeed, ``either $I \vDash \gamma$ or $I \vDash \lnot \gamma$'' should be syntactically represented by the trivial theorem 
\begin{equation}
ZF \vdash \gamma^I \lor \lnot \gamma^I
\end{equation}
Instead, what is really wanted is a notion like
\begin{equation}\label{realComplete}
ZF \vdash \gamma^I \text{ or } ZF \vdash \lnot \gamma^I
\end{equation}

As we suppose a base model $V$ for ZF, we are at hand with a interpretation for ZF itself. In this case, the notion of truth in a model is represented by ``either $I^V \vDash \gamma$ or $I^V \vDash \lnot \gamma$''. However, if our supposition of a model $V$ is not informed by any specific information other than $V \vDash ZF$, the interpretation works simply as the identity. Therefore, we return to the problem of establishing a notion as in (\ref{realComplete}).

Nonetheless, (\ref{realComplete}) is not achievable for any recursive extension of ZF:

\begin{incomplete}
There are formulas $\alpha \in \Lg_{PA}$ that are undecidable under any given interpretation I for any given recursive extension S of ZF.
\end{incomplete}

\begin{proof}

To prove this result we should reinternalize the provability predicate under the interpretation. Indeed, if we consider the theory $A = \{\varphi \mid S \vdash \varphi^I\}$, the statement ``x is an axiom'' becomes a semi-recursive predication. Thus it seems that we would not be able to internalize a truth predicate for this new theory.

The point is that we should not internalize the predicate directly for the theory A. Instead, we note that 

\begin{enumerate}
	\item ``x is a proof in S'' is recursive.
	\item ``x is I of a formula in A'' is recursive.

Thus

	\item ``x is a proof in S that ends with I of the y in A'' is recursive.
\end{enumerate}

We call $Pr(x, y)$ the representation of the last statement in A. 

Moreover, this is the specific proof predicate from which we construct the desired provability predicate used in G\"odel's incompleteness theorem. Thus, by applying Rosser's trick and the diagonal lemma, we obtain a formula $G$ that is undecidable in A. Therefore $G$ is undecidable under the interpretation I in the ZF extension S.\footnote{As indicated by Rodrigo Freire, this same result can be obtained by simply applying Craig's theorem on recursively enumerable sets of formulas being recursively axiomatizable \cite{craig1953axiomatizability}.}
\end{proof}

This theorem can be understood as a small extension of G\"odel's incompleteness theorem as we consider decidability under relations between theories. Moreover, it relates to results available in \textit{Satisfaction is not absolute} \cite{hamkins2013satisfaction}. In this article, Hamkins and Yang considered the idea that there may be arithmetical formulas $\rho$ that two models of ZF disagree -- even as these same models agree on what is the standard model for arithmetic. Although very insightful on interesting model constructions, it lacks a construction for the $\rho$ formula. This formula is obtained as the existential for a number representing a formula. In fact, exhibiting $\rho$ is not possible, for it would imply the inconsistency of ZF.

Put another way, we have shown a similar phenomena where the disagreement can be exhibited. To make it possible we considered a foundational view that accommodate our incomplete understanding of set theory and arithmetic. Thus, agreement about arithmetic is to be understood as having similar sets of arithmetical truths $\{\varphi \mid S \vdash \varphi^N\}$, being S some stage (or alternative stage) in the development of ZF. In this sense, there is a formula $\rho$ that would be true in some possible development of S and false in some other possible development of S.

\section{The coordination problem}

Lets consider the following fictional scenario for the development of set theory and arithmetic. There are two group of mathematicians that would decide about new axioms for set theory and arithmetic. The first one $G_s$ is responsible for set theory and the second $G_a$ for arithmetic. Lets further assume that $G_a$ agrees with the standard expansion of arithmetic in ZF ($A^{ZF}_N$ is considered valid for $G_a$). How should we frame the relation between the two groups?

Consider that $G_s$ have decided in favor of new axiom $A$ to set theory ZF. Notably, this would expand the set of arithmetic truth in $A^{ZF + A}_N$. Should $G_a$ consider this new set to be true? This being the general attitude towards arithmetic means that the standard reduction determine new truths for arithmetic. In what sense does, thus, the standard interpretation provides a foundation for new arithmetical truths? 
If we think the standard interpretation does this, it seems like we have simply assumed that arithmetic lives in set theory, without any further considerations. After all, this framework bounds the expansion of arithmetic truth to the expansion of set theoretic truth. Therefore, $G_a$ would have no authority over new arithmetic axioms after all.

In order to make room for this setting, one should consider that we have a better understanding on how arithmetic is reduced to set theory than we have for each of the theories. And, for this to work in general, we should consider the reduction of arithmetic in set theory \textbf{unrevisable}.

Very often we consider ourselves to have a good understanding on relations between things that we may not have a good understanding. This is the case between the translation of a sentence like ``Napoleon was an emperor''. We may have lots of doubts about the ontological status of the words used in this sentence and still be confident about how to translate it to Chinese.

Indeed, we may be more confident about the way we reduce arithmetic to set theory than about truth in those theories. Yet this is not sufficient to assume the unrevisability of the reduction relation. After some investigation over the concept of emperor, one has realized that the standard translation of emperor in Chinese does not really represents what English speakers refer as emperor. For instance, emperor is usually translated as `Huangdi' in Chinese, even though this word associate the monarch with his divinity. In English, although often associated with divinity, the word emperor can be used without divine association. So a more intricate description as `Napoleon was the non-divine man who ruled over the French empire' would be better (even if it is not practical).

If there are grounds for taking $N$ to be a privileged interpretation, those would be based on partial representations of arithmetic and set theory. Therefore, the idea that $N$ correctly works as a connection between the theories may be simply because we haven't advanced the theories enough. This would be a similar case if a Chinese working in the translation of a western modern history book has been translating `Emperor' as `Huangdi'. It seems perfectly fine if he believed this to be a general translation, given that the only time he applied the translation was for the `Emperor of the Holy Roman Empire'. But as he starts translating the Napoleonic period, the broader picture would force him to reconsider the generality of the translation. 

A different picture would be the case where the Chinese translator invented a language where $w$ means `blue chair'. Finding someone else using $w$ to refer to a red chair, he could correctly accuse the person to be using the word incorrectly. So this would be similar to the case where we consider arithmetic to be a definition inside set theory. But this being the case would imply that there is no foundational gain in studying the relation between the theories.

Whereas set theory has a foundational role for arithmetic, we may now consider that the standard interpretation is a good yet revisable set theoretic inspection over arithmetic. It is precisely because we assume the interpretation to be revisable that a foundational relation can be argued. As truth expands in both theories we evaluate conflicts and revise, if necessary, the interpretation to accommodate changes. A summary of the steps in the coordination of $G_a$ and $G_s$ can be:

\begin{enumerate}
	\item Every addition of axioms to one theory should provoke an inspection over the adequacy of the current interpretation of arithmetic in set theory.
	\item If a conflict emerges in the development of the theories, the two groups should meet to adjust the interpretation to prevent the conflict.
	\item The adequacy of an interpretation should have reasons for itself apart from accommodating the interpretation.
\end{enumerate}

As we see in step 2, the two communities should sit together and reevaluate the state of the reduction, if necessary. Hopefully, these conferences would hardly occur. But we should allow some independence to each group. For otherwise their development, especially on arithmetic, would turn to be by definition assumed in the development of the other. 

We have added some life to the grounding relation by allowing it to fail. Nevertheless, there is still a deeper problem. The following scenario is still possible:

\begin{enumerate}
	\item Each instance of the development allows one to fix the interpretation between the theories.
	\item And at least one of the extension of any state of arithmetic is not possibly interpreted in set theory.
\end{enumerate}

Allowing both of these possibilities weakens the edifice of the grounding relation. Each moment in the development of the theories is an incomplete stage in which we cannot anticipate the impossibility of reductions occurring further in the development of the theories. 
From (1), any addition to the theories allows one to find (or keep) an interpretation of arithmetic. However, from (2), finding those interpretations do not add to idea that arithmetic is indeed reducible to a given set theory. This scenario is possible, as we see in the next theorem.

\begin{noninterpretable}
Let S be a consistent extension of ZF and A an recursive extension of PA, then there is an extension $A^*$ of $A$ that is not interpretable in S.
\end{noninterpretable}

\begin{proof}
We extend the theory A by generating a sequence of theories that are not interpretable in S by a particular interpretation I. Being these theories compatible with each other, the union of them will not be interpretable in S. 

Let $A_0 = A$ and $\{I_1, I_2, \ldots\}$ an enumeration of all interpretations from the language of PA in the language of ZF. (abbreviation: $T \stackrel{J}{\leq} T'$ represents ``$T$ is interpreted in $T'$ by $J$'')
\begin{equation*}
A_{i+1} = 	\begin{cases}
				A_i + \lnot \varphi^{I_i} \text{ in case }
					A_i \stackrel{I_i}{\leq} S, 
					A_i \nvdash \varphi \text{ and }
					S \vdash \varphi^{I_i} \text{ for some } \varphi\\
				A_i + G, \text{ otherwise, being $G$ the G\"odel formula for } A_i
			\end{cases}
\end{equation*}
Let $A^* = \bigcup\limits_{i \in \omega} A_i$. We note that $A^*$ is a consistent extension of $A$. We prove that this theory is not interpretable in $S$. 

Suppose $A^*$ is interpretable by I in S, then $I = I_k$ for some natural number $k$. Notably, if a theory $T$ is interpreted in a theory $T'$, then any subtheory of $T$ is interpreted in $T'$ by the same interpretation. Thus the entire sequence of theories $\{A_1, A_2, \ldots\}$ is interpreted in $S$ by $I_k$. In particular, we have $A_k \stackrel{I_k}{\leq} S$ and that $A_{k+1} = A_i + \lnot \varphi^{I_i}$ or $A_{k+1} = A_k + G$ as in the definition. In the first case, we obtain the contradiction $S \vdash \varphi^{I_k}$ and $S \vdash \lnot \varphi^{I_k}$. In the second, we have either the contradiction $A_k \stackrel{I_k}{\nleq} S$ or that, for all $\alpha$, $A_i \vdash \alpha$ if, and only if, $S \vdash \alpha^I$. But, since $A_i \nvdash G$, it follows $S \nvdash G^{I_k}$ -- which, in turn, implies the contradiction $A_{k+1} \stackrel{I_k}{\nleq} S$.
\end{proof}

We proved that it is possible for the theories ZF and PA to part ways along the path of development. Although disturbing, this may simply account for the meaningfulness of the question about the reduction between the two theories. We have considered that we should conceive it to fail (even fatally, as in this case) in order to not take for granted that the reduction works. Note further that this pays tribute to the idea that by interpreting arithmetic in set theory we should inform something that was not simply given, i.e., that arithmetic lives in the realm of set theory. Nonetheless, the following result should challenge those who are still hopeful that the theories can possibly have a strong grounding relation:

\begin{probabilityzero}
For every consistent extension S of ZF, the probability that a random consistent extension of PA is interpretable in S is zero.
\end{probabilityzero}

\begin{proof}

Lets consider the set $\Sigma$ of consistent extensions of PA. 

From the incompletness theorem, there is a formula $G$ that is undecidable in PA. Thus both $PA + G$ and $PA + \lnot G$ are consistent.

Notably, this is still true for the addition of any finite number of new axioms $\alpha_1, \alpha_2, \ldots, \alpha_n$. There is a formula $G$ that is undecidable in $PA + \{\alpha_1, \alpha_2, \ldots, \alpha_n\}$. The process of adding axioms continues indefinitely.

Lets then index extensions with binary numbers in the following way:

\begin{enumerate}
	\item $A_0 = PA$.
	\item If $G$ is the G\"odel sentence in $A_i$, then $A_{i1}$ is $A_i + G$ and $A_{i0}$ is $A_i + \lnot G$. ($i1$ and $i0$ are the binary extension of the number $i$ with the digits $1$ and $0$)
	\item $\Sigma = \{A_n \mid n \in \omega\}$ is the set of finite extensions of PA.
\end{enumerate}

Note that each member of $\Sigma$ is a finite extension of PA. Now we include infinite extensions of PA in $\Sigma$. Let $\Pi$ be a set of $\bigcup C$, for each $C$ a subset-chain in $\Sigma$. The index for the members of $(\Pi \setminus \Sigma)$ can be describe by functions $\omega \longrightarrow \{1, 0\}$. Thus, as a simple consequence, the set of indexes of the extensions is in bijection with $P(\omega)$. 

Each theory with infinite sequences as indexes is indeed a different theory, for any difference in the sequence means that one theory has a a formula like $G$ and the other a formula $\lnot G$.

Nonetheless, the number of interpretations is trivially countable. Also, since the same interpretation cannot accommodate incompatible theories, there must be a extension of PA that is not interpretable in the extension S. Moreover, as we are comparing countable possible interpretable extensions with uncountable non-interpretable extensions, the probability of picking a interpretable extension is zero.
\end{proof}

We note that same can be obtained, even if the starting point includes all theorems of the set theory S under the interpretation. Indeed, we can include the theorems under a given interpretation at any point without interfering in the result.

\begin{coro}
For every consistent extension S of ZF, the probability that a random consistent extension of a S-standard version of PA ($\{\phi \mid S \vdash \phi^I\}$, being $I$ the current standard interpretation) is interpretable in S is zero.
\end{coro}

To prove this corollary, we need only to include the result in the theorem \ref{thm:incomplete} in the strategy of the last theorem.

Although extensions like $A^+$ are in general not interpretable in S, the process of generating these theories is internalizable in S. Therefore, we may say that $S$ proves the consistency statement for all these extensions. This is not enough to claim a proper foundational relation. The model construction emerging from this type of consistency proof is simply given by the existence of a model as in the Henkin canonical construction. Thus the foundational model one can generate provides little more information than saying that the theory is consistent\footnote{A general survey on the kind of information this Henkin model construction can be viewed in \textit{Translating non Interpretable Theories} \cite{freiretranslating}.}. Therefore, we should not consider those cases as a path to avoid the problem discussed in this section.

Lastly, lets consider a metaphor. Picture the situation in which we have the unstable equilibrium of a sphere on a hill with a very small slope. We would like to say that the appearence of equilibrium represents our intuitions about the reduction between the theories being correct. Indeed, we have put the sphere in a position that appears to be an equilibrium.  As the slope of the hill is very small, our perception of equilibrium works really well. However, even if it takes a long time, it will become evident that the interpretation of PA in ZF is not in equilibrium. 

The ideas developed in the present article, especially in theorem \ref{thm:noninterpretable}, bring attention to the fact that we are talking about an unstable hill. No matter how the sphere appears to be at rest, we know that eventually it will gain traction and fall. The project of using $N$ for grounding arithmetic truth is equivalent to finding the equilibrium peak of the hill. It seems to be a good project as we focus on the movement of the sphere -- but an analysis of the geography of the hill is already sufficient to conclude this hill to be unstable. We should not base our foundational investigations in guarantying that we have the correct interpretation. Instead, we should use the interpretations as it informs about arithmetic concepts and as it considers bundles of arithmetic formulas in the very expressive environment of set theory. The standard interpretation $N$ should not be taken to be \textbf{the relative grounding} of arithmetic in set theory.

\section{Final remarks}

Rather than manipulating models of PA, we considered interpretations of PA in ZF. Our goal was to accommodate the incomplete picture of the set theoretical metatheory into our analysis of the foundations of arithmetic. The ordinal interpretation expands what we may consider true in arithmetic: many undecidable formulas in PA become theorems when examined under the interpretation in ZF. This is a general phenomenon. For every well-founded interpretation of recursive extensions of PA in extensions of ZF, the interpreted version of arithmetic has more theorems than the original. This shows that studying arithmetic inside set theory can be significant. As one consider these interpretations, one is exploring the expansion of arithmetic truth and how the addition of bundles of axioms play out.

We continued by introducing the coordination problem. We considered two independent communities of mathematicians responsible for deciding over new axioms of ZF and PA. Using this setting, we studied the possibility of coordinating PA with PA’s interpretation in ZF. Nonetheless, we proved that it is possible to have extensions of PA not interpretable in a given set theory ST. Moreover, we considered a given recursive extension A of PA and supposed any extension of this theory to be equally likely. Here, we prove that the probability for an extension A be interpretable in ST is zero.

We have, therefore, set a framework to criticize the notion of grounding between theories such as arithmetic and set theory. However, this is not to be understood as a general criticism of the idea of using set theory to investigate foundational matters regarding arithmetic. Instead, we have soley shown that it  may be flawed to assume that set theory really provide grounds for arithmetic truth or a definitive description of the universe of numbers. Our suggestion is therefore to consider a foundational relation that aims primarily at conceptional clarification of the concepts involved in the studied theory. An expressively rich environment such as set theory is armed with tools to study arithmetical relations in wider settings than it would be possible without leaving its deductive apparatus.

\bibliography{referencias} 

\begin{thebibliography}{10}

\bibitem{craig1953axiomatizability}
William Craig.
\newblock On axiomatizability within a system.
\newblock {\em The journal of Symbolic logic}, 18(1):30--32, 1953.

\bibitem{ferreiros2015mathematical}
Jos{\'e} Ferreir{\'o}s.
\newblock {\em Mathematical knowledge and the interplay of practices}.
\newblock Princeton University Press, 2015.

\bibitem{freiretranslating}
Alfredo~Roque Freire.
\newblock Translating non interpretable theories.
\newblock {\em South America Journal of Logic}, 2020.

\bibitem{freire2012existence}
Rodrigo~A Freire.
\newblock Interpretation and truth in cantorian set theory.
\newblock {\em preprint}, 2015.

\bibitem{gabbay2010formalist}
Michael Gabbay.
\newblock A formalist philosophy of mathematics part i: Arithmetic.
\newblock {\em Studia Logica}, 96(2):219--238, 2010.

\bibitem{hamkins2012set}
Joel~David Hamkins.
\newblock The set-theoretic multiverse.
\newblock {\em The Review of Symbolic Logic}, 5(3):416--449, 2012.

\bibitem{hamkinsFreire2020biinterpretation}
Joel~David Hamkins and Alfredo~Roque Freire.
\newblock Bi-interpretation in weak set theories, 2020.

\bibitem{hamkins2013satisfaction}
Joel~David Hamkins and Ruizhi Yang.
\newblock Satisfaction is not absolute.
\newblock {\em arXiv preprint arXiv:1312.0670}, 2013.

\bibitem{kaye2007interpretations}
Richard Kaye and Tin~Lok Wong.
\newblock On interpretations of arithmetic and set theory.
\newblock {\em Notre Dame Journal of Formal Logic}, 48(4):497--510, 2007.

\bibitem{maddy2014second}
Penelope Maddy.
\newblock A second philosophy of arithmetic.
\newblock {\em The Review of Symbolic Logic}, 7(2):222--249, 2014.

\bibitem{putnam1967mathematics}
Hilary Putnam.
\newblock Mathematics without foundations.
\newblock {\em The Journal of Philosophy}, pages 5--22, 1967.

\bibitem{quine1964ontological}
Willard~V. Quine.
\newblock Ontological reduction and the world of numbers.
\newblock {\em The Journal of Philosophy}, 61(7):209--216, 1964.

\bibitem{tarski1953}
Alfred Tarski, Andrzej Mostowski, and Raphael~Mitchel Robinson.
\newblock {\em Undecidable theories}, volume~13.
\newblock Elsevier, 1953.

\bibitem{williamson2016absolute}
Timothy Williamson.
\newblock Absolute provability and safe knowledge of axioms.
\newblock {\em G{\"o}del’s Disjunction: The scope and limits of mathematical
  knowledge}, pages 243--252, 2016.

\end{thebibliography}
\bibliographystyle{plain}
\end{document}